\documentclass[11pt]{amsart}
\usepackage{amsbsy}
\usepackage{graphicx,epsfig,subfigure,psfrag}
\begin{document}

%%%%%%%%%%%%%%%%%%%%%%%%%%%%%%%%%%%%%%%%%%%%%%%%%%%%%%%%%%%%%%%%%%%%
% Theorem, definition, lemma, proposition, corollary and proof
%%%%%%%%%%%%%%%%%%%%%%%%%%%%%%%%%%%%%%%%%%%%%%%%%%%%%%%%%%%%%%%%%%%%
%%%%%%%%%%%%%%%%%%%%%%%%%%%%%%%%%%%%%%%%%%%%%%%%
\newtheorem{theorem}{Theorem}%[section]
\newtheorem{proposition}{Proposition}%[section]
\newtheorem{lemma}{Lemma}%[section]
\newtheorem{corollary}{Corollary}%[section]%%
\newtheorem{definition}{Definition}%[section]
\newtheorem{remark}{Remark}%[section]
\newtheorem{remarks}{Remarks}%[section]
%%%%%%%%%%%%%%%%%%%%%%%%%%%%%%%%%%%
%%%%%%%%%%%%%%%%%%%%%%%%%%%%%%%%%%%%%%%%%%%%%%
%%\newcommand{\be}{\begin{equation}}
%%\newcommand{\ee}{\end{equation}}
%%%%%%%%%%%%%%%%%%%%%%%%%%%%%%%%%%%%%%%%%%%%%%%
%%%%%%%%%%%%%%%%%%%%%%%%%%%%%%%%%%%%%%%%%%%%
\newcommand{\tex}{\textstyle}
%%\DeclareMathOperator{\Ind}{Ind}
%% \DeclareMathOperator{\Card}{Card}
%% \DeclareMathOperator{\Deg}{Deg}
%% \DeclareMathOperator{\dist}{dist}
%% \DeclareMathOperator{\Signature}{Signature}
%% \DeclareMathOperator{\MC}{MC}
%% \DeclareMathOperator{\sign}{sign}
%%  \DeclareMathOperator{\Int}{int}
%%%%%%%%%%%%%%%%%%%%%%%%%%%%%%%%%%%%%%%%%%%%
%%%%%%%%%%%%%%%%%%%%%%%%%%%%%%%%%%%%%%%%%%%%
\numberwithin{equation}{section} \numberwithin{theorem}{section}
\numberwithin{proposition}{section} \numberwithin{lemma}{section}
\numberwithin{corollary}{section}
\numberwithin{definition}{section} \numberwithin{remark}{section}
%%%%%%%%%%%%%%%%%%%%%%%%%%%%%%%%%%%%%%%%%%%%
%%%%%%%%%%%%%%%%%%%%%%%%%%%%%%%%%%%%%%%%%%%%
\newcommand{\ren}{\mathbb{R}^N}
\newcommand{\re}{\mathbb{R}}
\newcommand{\dyle}{\displaystyle}
\newcommand{\n}{\nabla}
\newcommand{\p}{\partial}
\newcommand{\iy}{\infty}
\newcommand{\pa}{\partial}
\newcommand{\fp}{\noindent}
\newcommand{\ms}{\medskip\vskip-.1cm}
\newcommand{\mpb}{\medskip}
%%%%%%%%%%%%%%%%%%%%%%%%%%%%%%%%%%%%%%%%%%%%%%%%%
\newcommand{\AAA}{{\bf A}}
\newcommand{\BB}{{\bf B}}
\newcommand{\CC}{{\bf C}}
\newcommand{\DD}{{\bf D}}
\newcommand{\EE}{{\bf E}}
\newcommand{\FF}{{\bf F}}
\newcommand{\GG}{{\bf G}}
\newcommand{\oo}{{\mathbf \omega}}
\newcommand{\Am}{{\bf A}_{2m}}
\newcommand{\CCC}{{\mathbf  C}}
\newcommand{\II}{{\mathrm{Im}}\,}
\newcommand{\RR}{{\mathrm{Re}}\,}
\newcommand{\eee}{{\mathrm  e}}

\newcommand{\cN}{{\mathcal{N}}}

%%%%%%%%%%%%%%%%%%%%%%%%%%%%%%%%%%%%%%%%%%%%%%%%%%%%%%%%%%%%%%%%%%%%%%% L^2\rho...
\newcommand{\LL}{L^2_\rho(\ren)}
\newcommand{\LLL}{L^2_{\rho^*}(\ren)}
%%%%%%%%%%%%%%%%%%%%%%%%%%%%%%%%%%
%%%%%%%%%%%%%%%%%%%%%%%%%%%%%%%%%%%%%%%%%%%%%%%%%%%%
\renewcommand{\a}{\alpha}
\renewcommand{\b}{\beta}
\newcommand{\g}{\gamma}
\newcommand{\G}{\Gamma}
\renewcommand{\d}{\delta}
\newcommand{\D}{\Delta}
\newcommand{\e}{\varepsilon}
\newcommand{\var}{\varphi}
\newcommand{\lll}{\l}
\renewcommand{\l}{\lambda}
\renewcommand{\o}{\omega}
\renewcommand{\O}{\Omega}
\newcommand{\s}{\sigma}
\renewcommand{\t}{\tau}
\renewcommand{\th}{\theta}
\newcommand{\z}{\zeta}
\newcommand{\wx}{\widetilde x}
\newcommand{\wt}{\widetilde t}
\newcommand{\noi}{\noindent}
 %%%%%%%%%%%%%%%%%%%%%%%%%%%%%%%%%%%%%%%%%%%
\newcommand{\uu}{{\bf u}}
\newcommand{\xx}{{\bf x}}
\newcommand{\yy}{{\bf y}}
\newcommand{\zz}{{\bf z}}
\newcommand{\aaa}{{\bf a}}
\newcommand{\cc}{{\bf c}}
\newcommand{\jj}{{\bf j}}
\newcommand{\ggg}{{\bf g}}
\newcommand{\UU}{{\bf U}}
\newcommand{\YY}{{\bf Y}}
\newcommand{\HH}{{\bf H}}
\newcommand{\GGG}{{\bf G}}
\newcommand{\VV}{{\bf V}}
\newcommand{\ww}{{\bf w}}
\newcommand{\vv}{{\bf v}}
\newcommand{\hh}{{\bf h}}
\newcommand{\di}{{\rm div}\,}
\newcommand{\ii}{{\rm i}\,}
\def\I{{\rm Id}}
%%%%%%%%%%%%%%%%%%%%%%%%%%%%%%%%%%
%%%%%%%%%%%%%%%%%%%%%%%%%%%%%%%%%%%%%   VAG, NEW
\newcommand{\inA}{\quad \mbox{in} \quad \ren \times \re_+}
\newcommand{\inB}{\quad \mbox{in} \quad}
\newcommand{\inC}{\quad \mbox{in} \quad \re \times \re_+}
\newcommand{\inD}{\quad \mbox{in} \quad \re}
\newcommand{\forA}{\quad \mbox{for} \quad}
\newcommand{\whereA}{,\quad \mbox{where} \quad}
\newcommand{\asA}{\quad \mbox{as} \quad}
\newcommand{\andA}{\quad \mbox{and} \quad}
\newcommand{\withA}{,\quad \mbox{with} \quad}
\newcommand{\orA}{,\quad \mbox{or} \quad}
\newcommand{\atA}{\quad \mbox{at} \quad}
\newcommand{\onA}{\quad \mbox{on} \quad}
\newcommand{\ef}{\eqref}
\newcommand{\mc}{\mathcal}
\newcommand{\mf}{\mathfrak}

\newcommand{\Ge}{\G_\e}
\newcommand{\Hn}{ H^{1}(\Rn)}
\newcommand{\Wn}{W^{1,2}(\Rn)}
\newcommand{\Wan}{W^{\frac{\a}{2},2}(\Rn)}
\newcommand{\Wa}{W^{\frac{\a}{2},2}(\R)}
\newcommand{\intn}{\int_{\Rn}}
\newcommand{\intR}{\int_\R}
\newcommand{\ie}{I_\e}
\newcommand{\nie}{\n \ie}
\newcommand{\gie}{I_\e'}
\newcommand{\ies}{I_\e''}
\newcommand{\ios}{I_0''}
\newcommand{\ip}{I'_0}

\newcommand{\zex}{z_{\e,\rho}}
\newcommand{\wex}{w_{\e,\xi}}
\newcommand{\zer}{z_{\e,\rho}}
\newcommand{\wer}{w_{\e,\rho}}
\newcommand{\dzex}{{\dot{z}}_{\e,\rho}}
\newcommand{\bE}{{\bf E}}
\newcommand{\bu}{{\bf u}}
\newcommand{\bv}{{\bf v}}
\newcommand{\bz}{{\bf z}}
\newcommand{\bw}{{\bf w}}
\newcommand{\bo}{{\bf 0}}
\newcommand{\bp}{{\bf \phi}}
\newcommand{\up}{\underline{\phi}}
\newcommand{\bh}{{\bf h}}

\newcommand{\ssk}{\smallskip}
\newcommand{\LongA}{\quad \Longrightarrow \quad}
%%%%%%%%%%%%%%%%%%%%%%%%%%%%%%%%
%%%%%%%%%%%%%%%%%%%%%%%%%%%%%%%%%%
\def\com#1{\fbox{\parbox{6in}{\texttt{#1}}}}
%%%%%%%%%%%%%%%%%%%%%%%%%%%%%%%%%%
%%%%%%%%%%%%%%%%%%% From Paper1
\def\N{{\mathbb N}}
\def\A{{\cal A}}
\newcommand{\de}{\,d}
\newcommand{\eps}{\varepsilon}
\newcommand{\be}{\begin{equation}}
\newcommand{\ee}{\end{equation}}
\newcommand{\spt}{{\mbox spt}}
\newcommand{\ind}{{\mbox ind}}
\newcommand{\supp}{{\mbox supp}}
\newcommand{\dip}{\displaystyle}
\newcommand{\prt}{\partial}
\renewcommand{\theequation}{\thesection.\arabic{equation}}
\renewcommand{\baselinestretch}{1.1}
%%%%%%%%%%%%%%%%%%%%%%%%%%%%%%%%%%%%%%%%%%%%%%%
\newcommand{\Dm}{(-\D)^m}

\newenvironment{pf}{\noindent{\it
Proof}.\enspace}{\rule{2mm}{2mm}\medskip}

\newcommand{\lapa}{(-\Delta)^{\alpha/2}}

%%%%%%%%%%%%%%%%%%%%%%%%%
\title
%%%%%
%%%%%%%%%%%%%%%%%%%%%%%%%
%%%%
{\bf Existence of solutions for a system of coupled nonlinear
stationary bi-harmonic Schr\"{o}dinger equations}

\author{P.~\'Alvarez-Caudevilla, E.~Colorado, and V.~A.~Galaktionov}

\address{Departamento de Matem\'aticas, Universidad Carlos III de Madrid,
Av. Universidad 30,
28911 Legan\'es (Madrid), Spain}
\email{pacaudev@math.uc3m.es}

\address{Departamento de Matem\'aticas, Universidad Carlos III de Madrid,
Av. Universidad 30, 28911 Legan\'es (Madrid), Spain. And Instituto de Ciencias Matem\'aticas, (ICMAT, CSIC-UAM-UC3M-UCM), C/Nicol\'as Cabrera 15,
28049 Madrid, Spain.}
\email{ecolorad@math.uc3m.es, eduardo.colorado@icmat.es}

\address{Department of Mathematical Sciences, University of Bath,
 Bath BA2 7AY, UK}

\email{masvg@bath.ac.uk}

%\keywords{Nonlinear Bi-Harmonic Schr\"odinger Equations, Standing Waves, Critical Point Theory.\\
%{\it \indent 2010 Mathematics Subject Classification} 35G20, 35Q55, 35J50, 35B38.}

\thanks{The first and third authors have been partially supported by the Ministry of Economy and Competitiveness of
Spain under research project MTM2012-33258.}

\date{\today}

%%%%%%%%%%%%%%%%%%%%%%%%%%%
\maketitle
\begin{abstract}
We obtain existence and multiplicity results for the solutions of a class of coupled
semilinear bi-harmonic Schr\"{o}dinger equations. Actually,  using the classical
Mountain Pass Theorem and minimization techniques, we prove the existence of critical points of the associated functional constrained on the \emph{Nehari manifold}.

Furthermore, we show that using the so-called
\emph{fibering method} and the \emph{Lusternik-Schnirel'man theory} there exist infinitely
many solutions, actually a countable family of critical points, for such a semiliner bi-hamonic Schr\"{o}dinger system under
study in this work.
\end{abstract}


%%%%%%%%%%%%%%%%%%%%%%%%%%%%%%%%%%%%%%%%%%%%%%%%%%%%%%%%%%%%%%%%%%%%
%%%%%%%%%%%%%%%%%%%%%%%%%%%%%%%%%%%%%%%%%%%%%%%%%%%%%%%%%%%%%%%%%%%%
%%%%%%%%%%%%%%%%%%%%%%%%%%%%%%%%%%%%%%%%%%%%%%%%%%%%%%%%%%%%%%%%%%%%
\section{Introduction}\label{sec:intro}
%%%%%%%%%%%%%%%%%%%%%%%%%%%%%%%%%%%%%%%%%%%%%%%%%%%%%%%%%%%%%%%%%%%%
%%%%%%%%%%%%%%%%%%%%%%%%%%%%%%%%%%%%%%%%%%%%%%%%%%%%%%%%%%%%%%%%%%%%
%%%%%%%%%%%%%%%%%%%%%%%%%%%%%%%%%%%%%%%%%%%%%%%%%%%%%%%%%%%%%%%%%%%%
This work is devoted to the analysis of  solutions that solve the
following coupled nonlinear stationary bi-harmonic Schr\"{o}dinger
system (BNLSS)
\be
\label{s0} \left\{\begin{array}{ll}  \D^2 u_1
+\l_1 u_1= \mu_1 u_1^{2\s+1}  +\b |u_1|^{\s -1} u_1 |u_2|^{\s+1}
\ssk \\ \D^2 u_2 +\l_2 u_2= \mu_2 u_2^{2\s+1}  +\b |u_1|^{\s+1}
|u_2|^{\s-1}u_2\end{array}
\right. \ee where $\l_j,\mu_j>0$, $u_j\in W^{2,2}(\re^N)$ with $j=1,2$, $\b$ denotes a
real parameter and $x\in\re^N$, with $N=2,3$ (for physical
purposes).

To simplify the computations in this work we assume $\sigma=1$, hence we will study the system
\be
\label{s1}
\left\{\begin{array}{ll} \D^2 u_1 +\l_1 u_1= \mu_1 u_1^{3} +\b   u_2^{2}u_1
 \ssk
\\ \D^2 u_2 +\l_2 u_2= \mu_2 u_2^{3} +\b u_1^{2}u_2\end{array} \right. \ee
which has been
analyzed in the context of stability of solitons in magnetic
materials when effective quasi-particle mass becomes infinite.
Moreover, note that  system \eqref{s1} appears after assuming the
bi-harmonic nonlinear Schr\"{o}dinger equation (BNLSE) of the form
\be
\label{bse}
iW_t -\D^2 W+\b |W|^{2\s} W=0,
\ee
where $i$ denotes the imaginary unit. Then, if $W$ is the sum of two right and left-hand polarized waves $a_1 W_1$ and
$a_2 W_2$, where $a_1,a_2\in \re$, the preceding equation \eqref{bse} provides us with the following
coupled nonlinear bi-harmonic Schr\"{o}dinger system
\be
\label{s2}
\left\{\begin{aligned} &  iW_{1,t} - \D^2 W_1+  |a_1 W_1 + a_2 W_2|^{2\s} W_1= 0\\ &
 iW_{2,t} - \D^2 W_2+ |a_1 W_1 + a_2 W_2|^{2\s} W_2= 0\end{aligned} \right.
\ee where $W_{j,t}=\frac{\p W_j}{\p t}$, $j=1,2$. For this problem
we look for standing waves or finite-energy waveguide solutions of
the form $$W_j(t,x)=e^{i\l_j t} u_j(x),\quad j=1, \,2,
  $$
  where
$\l_j>0$ and $u_i$ are real value functions, which solve the system
\eqref{s1}. Rearranging terms in \eqref{s2} one can easily see that $u_j$ solve  the stationary system \eqref{s1}.

Problem \eqref{s1} is the \emph{bi-harmonic} version of a similar
one studied, among others, in \cite{ac,ac2,mmp,linwei,sirakov}
where a non-linear system of coupled nonlinear Schr\"odinger equations (NLSE)
of the form
\be
\label{nls}
\left\{\begin{array}{ll} -\D u_1 +\l_1 u_1= \mu_1 u_1^3 +\b  u_2^2u_1\\
-\D u_2 +\l_2 u_2= \mu_2 u_2^3 +\b u_1^2u_2 \end{array} \right.
\ee
with direct applications to nonlinear optics, Bose-Einstein condensates, etc, was considered. See also \cite{acr} where a linearly coupled system was considered and note that in \cite{col} system \eqref{nls} was studied in the one-dimensional case dealing with the fractional Schr\"odinger operator $(-\D)^{s}+\,$Id, $\frac 14<s\le 1$.

Here, we assume that the solutions belong to the Sobolev space $E=W^{2,2}(\re^N)$, endowed  with the scalar product and norm
\be
\label{embso}
\tex{\left\langle u,v \right\rangle_j := \int_{\re^N} \D u \cdot \D v
%+ \int_{\re^N}�\nabla u \cdot  \nabla v
+ \lambda_j\int_{\re^N} uv,\quad\|u\|_j^2=\left\langle u,u \right\rangle_j,\: j=1,2.}
\ee

Also, we define $\mathbb{E}=E\times E$, and the elements in $\mathbb{E}$ will be denoted
by $\bu=(u_1,u_2)$; as a norm in $\mathbb{E}$ we will take
$$\|\bu\|^2=\|u_1\|_1^2+\|u_2\|_2^2.$$

Moreover, we denote $H$ as the space of radially symmetric functions in $E$, and $\mathbb{H}=H\times H.$
For $u\in E$, respectively, $\bu\in \mathbb{E}$, we set
\begin{eqnarray}
\label{mainfun}
 I_j(u) &=& \tfrac 12 \int_{\mathbb{R}^N} \left(|\D
u|^2+\l_j u^2\right)\,{\mathrm d} x -\tfrac{1}{4}\,\mu_j
\int_{\mathbb{R}^N} u^{4}\,{\mathrm d} x,\\ F(\bu) &=&
\tfrac{1}{4}\,
\int_{\mathbb{R}^N} \left(\mu_1 u_1^{4} +\mu_2
u_2^{4}\right)\,{\mathrm d} x,\\ G(\bu)
 %%%% }
 &=& G(u_1,u_2)= \tfrac 12
 \int_{\mathbb{R}^N} |u_1|^2|u_2|^2 \,{\mathrm d} x,
  \ssk
\\ \label{mainfun0}\mathcal{J}(\bu)
&=&\mathcal{J}(u_1,u_2)= I_1(u_1)+I_2(u_2) - \b\, G(u_1,u_2)\\ &=&
\tfrac 12 \|\bu \|^2- F(\bu) -\b\,G(\bu).
\end{eqnarray}
\begin{remark}\label{rem:compact}
We recall  a well known result  about continuous  Sobolev
embedding (see, for instance, \cite{adams,Lions-JFA82}),
\be
\label{continuous}
\tex{E \hookrightarrow L^{p}(\re^N),\quad\hbox{with}\quad
1\le p \le p^*,}
\ee
which are compact replacing $E$ by the radial subspace $H$ and  if in
 addition $2\le N$ and $p<p^*$ (see \cite{Lions-JFA82}). Besides, we recall here the definition of the critical exponent
\be
\label{sobo12} \tex{p^*= \frac {2N}{N-4} \mbox{ if } N \ge
5, \quad \hbox{and}\quad
     \quad p^*=\iy\,\,\,\mbox{for}\,\,\,N=1,\, 2,\, 3,\, 4.}
     \ee
\end{remark}

We observe that, by \eqref{continuous}, the functional $\mathcal{J}$
is well defined since $F$, $G$ make sense for  $4\leq
p^*\Leftrightarrow N\le 8$, moreover, for $2\le N< 8$ we have that $F, G$ are compact on $\mathbb{H}$. Furthermore, it is easy to prove that
the functional $\mathcal{J}$ associated to \eqref{s1} is $\mc{C}^1$.
%%%%%%%%%%%%%%%%%%%%%%%%%%%%%%%%%%%%%%%%%%%%%%%%%%%%%%

\subsection{Main results}

%%%%%%%%%%%%%%%%%%%%%%%%%%%%%%%%%%%%%%%%%%%%%%%%%%%%%

We basically ascertain existence and multiplicity results for the system \eqref{s1}.
To do so, we say that $\bu=(u,v)\neq (0,0)=\bo$ is a non-trivial solution of \eqref{s1} if $\bu$ is a critical point of $\mathcal{J}$.

We now state the definitions and  differences between bound and ground states (non-trivial solutions).
     %To simplify the exposition of these already known results we write $H_{{\rm rad}}^2(\re^N)=H^2$.
     \begin{definition}
     \label{bou}
     {\rm $\bu\in \mathbb{E}$ is a} non-trivial bound state
      {\rm of the system \eqref{s1} if $\bu$ is a non-trivial critical point of the functional $\mc{J}$,
     i.e., $\mc{J}'(\bu)=0$. Moreover, a bound state $\bw$ such that its energy is minimal among all the non-trivial
     bound states, namely,
     $$\mc{J}(\bw )=\min\{\mc{J}(\bu)\,;\, \bu\in \mathbb{E} \setminus\{(0,0)\},\; \mc{J}'(\bu)=0\},$$
     is called} ground state {\rm for the system \eqref{s1}.}
     \end{definition}
Thus, we first analyse (Section \ref{Sec2})
the bi-harmonic non-linear Schr\"{o}dinger equation
\be
\label{exnls}
 \tex{ \D^2 u_\l +  \l u_\l - |u_\l |^{2\s} u_\l= 0,}
 \ee
with\footnote{Note that the results we are going to prove can be extended easily for  $\sigma <\frac{4}{4-N}$.} $\s =1$, with the associated functional
\begin{equation}
\label{funfb}
\mc{F}_\l(u):=\tfrac{1}{2} \int\limits_{\re^N} (|\Delta u|^2 + \l u^2)  -  \tfrac{1}{2\s+2} \int\limits_{\re^N}
    |u|^{2\s+2}.
\end{equation}
 For this equation we show the existence of the ground state
through an argument based on the so-called \emph{fibering method}. This is basically an alternative methodology
to the bifurcation theory or critical point theory in obtaining information about the number of solutions
for certain differential equations with a variational structure; see \cite{DraPoh,Poh0, PohFM} for several examples and further details.

For equation \eqref{exnls} the solutions have an exponential decay at infinity and using the topological method
due to Lusternik-Schnirel'man we ascertain that there exists an infinite  countable number of solutions.

Furthermore, in Section \ref{Sec3} we prove the existence of solutions or critical points for the bi-harmonic non-linear Schr\"{o}dinger system \eqref{s1}.
Thus,  in order to find critical points of the functional $\mc{J}$ we set
     \be
     \label{setfun}
     \tex{\Psi(\bu)=(\mc{J}'(\bu) \,| \, \bu )= \|\bu\|^2  -4 F(\bu)- 4\b G(\bu),}
     \ee
 Then, we show that when the parameter $\b$ is less than the minimum among the  two following Sobolev constants
\[ S_j^2  =  \dyle\inf_{\varphi\in E\setminus\{0\}}
\frac{\|\varphi\|_k^2}{\int_{\re^N} U_j^2\varphi^2},\mbox{ with  } k,j=1,\, 2,\quad k\neq j,
\]
 where $U_j$ are the ground states of the equations
\[\D^2 u +  \l_ju = \mu_ju^3,\]
such that system \eqref{s1} possesses two semi-trivial solutions
$$\bu_1=(U_1,0), \quad \bu_2=(0,U_2),$$
which are strict local minima. On the other hand if the coupling parameter $\b$ is bigger than the maximum value among those Sobolev constants,
i.e.,
\[\b>\max\{S_1^2,S_2^2\},\]
then the semi-trivial solutions are saddle points of the associated functional $\mc{J}$ denoted by \eqref{mainfun}, to the system \eqref{s1},
constrained on the \emph{Nehari manifold}
\be
\label{Neh}
\mc{N}:=\{ \bu\in \mathbb{H} \setminus\{\bo\}\,;\, \Psi(\bu)=0\}.
\ee
Consequently, we are able to prove that
when the functional satisfies the Mountain Pass geometry, if $\b <\min\{S_1^2,S_2^2\}$ then there exists a critical point $\bu^*$ such that
\[\mc{J}(\bu^*)>\max\{\mc{J}(\bu_1),\mc{J}(\bu_2)\}.\]
However, when $\b> \max\{S_1^2,S_2^2\}$ we ascertain that a global minimum $\widetilde{\bu}$ exists for the function $\mc{J}$ on the Nehari manifold $\mc{N}$, such that
$$\mathcal{J}(\widetilde{\bu})<\min\{\mathcal{J}(\bu_1) ,\mathcal{J}(\bu_2)\}.$$
It is not difficult to prove that $\cN$ is a natural restriction following the same kind
of arguments as in \cite{ac2}.  We include the computations for the reader's convenience.

 \begin{proposition}\label{pr:ac}
 $\bu \in \mathbb{H}$ is a non-trivial critical point of $\mathcal{J}$ if and only if
 $\bu\in \cN$ and is a constrained critical point of $\mathcal{J}$ on $\cN$.
 \end{proposition}
\begin{proof}
For any $\bv\in \mathbb{H}\setminus \{\bo\}$,  one has that $$
t\bv \in \cN\quad \Longleftrightarrow\quad \|\bv\|^2= t^2\left[
4F(\bv)+4\b G(\bv)\right]. $$ As a consequence, for all $\bv\in
\mathbb{H}\setminus \{\bo\}$, there exists a unique $t>0$ such
that $t\bv \in \cN$. Moreover, since $F,G$ are homogeneous of
degree $4$, there exists $\rho>0$ such that
\begin{equation}\label{eq:M1}
\|\bu\|^2\geq \rho,\qquad \forall\, \bu\in \cN.
\end{equation}
Furthermore, from \eqref{eq:M1} it follows that
 \begin{equation}\label{eq:M2}
 (\Psi'(\bu)\mid \bu)=-2\|\bu\|^2< 0,\qquad \forall\,\bu\in \cN.
\end{equation}
From \eqref{eq:M1} and \eqref{eq:M2}, we infer that $\cN$ is a
smooth complete manifold of co-dimension one in $\mathbb{E}$.
Moreover, if
 $\bu\in \cN$ is a critical point of $\mathcal{J}$ constrained on $\cN$,  there exists $\o\in \mathbb{R}$ such that
 $$
 \mathcal{J}'(\bu)= \omega \Psi'(\bu).
 $$
 Then one finds $\Psi(\bu)=( \mc{J}'(\bu)\mid \bu)=\omega (\Psi'(\bu)\mid \bu)$. Since
 $\Psi(\bu)=0$ $\forall\,\bu\in \cN$, while thanks to \eqref{eq:M2} it follows $(\Psi'(\bu)\mid \bu)<-2\rho<0$, we infer that
 $\o=0$ and, thus,  $ \mc{J}'(\bu)= 0$.
\end{proof}
 %%In conclusion, we have proved the following result.

\begin{remarks}\label{rem:PS}
\begin{enumerate}
\item
Note that due to  the definition of $\cN$,
 it follows that
\begin{equation}\label{eq:F}
\|\bu \|^2= 4 F(\bu) +4\b G(\bu).
\end{equation}
Substituting into $\mathcal{J}$, we get
\begin{equation}\label{eq:M3}
\mathcal{J}(\bu)=\tfrac{1}{4} \|\bu \|^2,\qquad \forall\,\bu\in \cN,
\end{equation}
or equivalently,
\begin{equation}\label{eq:M4}
\mathcal{J}(\bu)= F(\bu) +\b G(\bu),\qquad \forall\,\bu\in \cN.
\end{equation}
Then, \eqref{eq:M3} jointly with  \eqref{eq:M1} imply that there
exists $C>0$ such that
\begin{equation}\label{eq:bound}
\mathcal{J}(\bu)\geq C>0,\qquad  \forall\,\bu\in \cN.
\end{equation}
 As a consequence, the main relevant fact of working on the Nehari manifold is that
 $\mathcal{J}$ is bounded from below on $\cN$, so one can try to minimize on it.
\item Concerning  the Palais-Smale (PS) condition, it is not difficult to prove it by
 taking into account the compact embedding of $H$ into $L^p(\mathbb{R}^N)$ for any $1< p<p^*$ and $N\ge 2$;
  see Remark \ref{rem:compact} and Lemma \ref{lem:PS}.
\end{enumerate}
\end{remarks}
Moreover, for system \eqref{s1} we apply again the fibering method and the Lusternik-Schnirel'man analysis to show the existence of an infinite countable number of
solutions. Indeed, we prove through a fibering method argument that when the parameter
\[\b >-\sqrt{\mu_1\mu_2},\]
system \eqref{s1} possesses at least a ground state solution. Note that as far as we know this is the first time the fibering has been applied to
the analysis of systems.
Moreover, Lusternik-Schnirel'man theory provides us again with the
existence of infinitely many solutions for system \eqref{s1}.

The outline of the paper is: the analysis of the bi-harmonic equation \eqref{exnls} is analyzed in Section \ref{Sec2}. In Section \ref{Sec3}, we study the existence of critical points of the functional corresponding to the system \eqref{s1}. Finally, we prove the existence of infinitely many solutions for the previous system \eqref{s1} in Section \ref{Sec4}.
%%%%%%%%%%%%%%%%%%%%%%%%%%%%%%%%%%%%%%%%%%%%%%%%%%%%%%%%%%%%%%%%%%%%%
%%%%%%%%%%%%%%%%%%%%%%%%%%%%%%%%%%%%%%%%%%%%%%%%%%%%%%%%%%%%%%%%%%%%%
%%%%%%%%%%%%%%%%%%%%%%%%%%%%%%%%%%%%%%%%%%%%%%%%%%%%%%%%%%%%%%%%%%%%%
%%%%%%%%%%%%%%%%%%%%%%%%%%%%%%%%%%%%%%%%%%%%%%%%%%%%%%

\section{Bi-harmonic non-linear Schr\"{o}dinger equation}\label{Sec2}

%%%%%%%%%%%%%%%%%%%%%%%%%%%%%%%%%%%%%%%%%%%%%%%%%%%%%

\noindent During the last decades many works have been orientated
to the analysis of nonlinear Schr\"{o}dinger equations (NLSE) (cf.
\cite{ac2}) of the form
\be
\label{nls22} \tex{iw_t +\D w+ |w|^{2} w=0,\quad w(0,x)=w_0(x)\in
W^{1,2}(\re^N),} \ee
especially because of its strong applications to
several areas of Physics, such as nonlinear optics and
Bose--Einstein condensates. In particular, for equation
\eqref{nls22} it is known that it possesses solutions which become
singular, i.e., blows-up in finite time. In other words, these
solutions exist in $W^{1,2}(\re^N)$ over a finite time  interval such that
$$\lim_{t\to T} \|w\|_{W^{1,2}}=\infty.$$ Equation \eqref{nls22} is
the canonical model for propagation of intense laser beams in bulk
medium with \emph{Kerr nonlinearity}. However, that equation can
be generalized with a general power-law nonlinearity by
 \be
\label{nls23}
\tex{iw_t +\D w+ |w|^{2\s} w=0,\quad w(0,x)=w_0(x)\in W^{1,2}(\re^N),}
\ee
whose analysis has been primarily focused on the existence of standing waves, their stability and the global or blow-up existence of solutions.
Indeed, depending on the dimension one can assure that when $\s N<2$ (subcritical NLSE) there is existence of solutions globally. However,
in the critical case, i.e., equation \eqref{nls22} with $\s N=2$ for $N=2$, and the supercritical case $\s N>2$ the solutions might blow-up and the
standing waves are unstable.

Furthermore, and important in our analysis there are some studies focusing on a
%a bi-harmonic NLS equation (BNLSE)
BNLSE of the form
\be
\label{bnls22}
\tex{iw_t -\D^2 w+ |w|^{2\s} w=0,\quad w(0,x)=w_0(x)\in E,}
\ee
similar to \eqref{bse}. See for instance \cite{BFM,FIP} where blow-up solutions are considered and others such as \cite{MXZ,Pau1,Pau2} where existence of global solutions
is analysed. In relation with the existence of solutions for the BNLSE equation \eqref{bnls22}, it was proved  in \cite{FIP} that
some of the properties and characteristics for the NLSE \eqref{nls23} can be extended to BNLSE of the form \eqref{bnls22},
with the obvious transformations. In particular, looking for standing wave solutions for \eqref{bnls22} of the form
\be
\label{standw}
w(t,x)=e^{i\l t} u_\l (x),
\ee
such that $u$ is a radially symmetric ground state solution satisfying the equation
\be
\label{radnls}
 \tex{ \D^2 u_\l +  \l u_\l - |u_\l |^{2\s} u_\l= 0,}
 \ee
 for which there exist non-trivial solutions (radially symmetric ground state solutions) if $\s <\frac{4}{4-N}$.

 Note that due to the expression of equation \eqref{bnls22} we are working on the \emph{focusing} BNLSE. For the NLSE \eqref{nls22} it is
said that the equation is \emph{focusing or defocusing} when the
diffraction and nonlinearity work one against or with each other.
Therefore, since the Laplacian is a negative operator this would
correspond to positive or negative nonlinearity. However, the
bi-Harmonic operator is positive so that \eqref{bnls22} is called
a \emph{focusing} BNLSE.

Moreover, depending on $\s$ and the dimension (see \cite[Section 4]{FIP}) one has that:
\begin{itemize}
\item If $\s N < 4$ the standing wave solutions \eqref{standw} are stable;
\item In the critical case $\s N=4$ there exist global solutions if the $L^2$-norm for the initial condition $w_0$ is bounded. This suggests that
\emph{a priori} bounds imply global existence;
\item However, the critical exponent/dimension $\s N=4$ creates singularity formation in the \emph{focusing} BNLSE \eqref{bnls22}.
\end{itemize}
Another very important aspect shown in \cite{FIP} about the BNLSE \eqref{bnls22} is that the solutions or critical points
of the associated functional, once they exist, oscillate.
On the contrary to what happens for the NLSE \eqref{nls22}
in which the ground state is positive and radially symmetric.
Indeed, the \emph{ground states} or non-trivial solutions of
minimal energy for the equation \eqref{radnls} are neither
positive nor monotonic.

It should be pointed out that even for bounded domains the issues concerning the
positivity of the first eigenfunction of a bi-harmonic operator are not straightforward. Indeed, there are only
partial results about this problem having the positivity of the first eigenfunction in a ball or for
certain perturbations of the bi-harmonic operators or of the domain; see \cite{HcS} for further details and references therein.

Also, using the two-scale WKBJ approximation method, it is possible
to approximate the solutions of the PDE \eqref{radnls} assuming
that the radially symmetric ground state can be approximated by
$u(r)=e^{g(r)}$.

Then, analysing the leading order terms one can
ascertain the asymptotic behaviour of the solutions.
Namely, scaling out the parameter $\l>0$
for  $u=u(r)$, with $r = |x| \ge 0$, so that
$$u(r)=\l^{-1/2\s} u(\l^{-1/4} r),$$
we obtain,
 \be
 \label{inf1}
 %%% \begin{matrix}
  \tex{
 \D^2 u \equiv u^{(4)} + \frac{2(N-1)}{r} u'''
    + \frac{2(N-1)(N-3)}{r^2} u'' - \frac{(N-1)(N-3)}{r^3} u'=-u + |u|^{2\s}u\, ,
 }
 \ee
 with $\s<\frac{4}{4-N}$.
Next, using  a standard algebraic-exponential pattern of the form $ u(r) \approx r ^\d \,\eee^{a r}$ (as $r\to \infty$) in \eqref{inf1} leads easily to the
following characteristic equation
 \be
 \label{inf111}
  \tex{
a^4=-1 \andA \d= - \frac{N-1}2.
 }
 \ee
To be precise, note that the first equation in \eqref{inf111} comes from the homogeneity of the leading terms in \eqref{inf1}. The second equality in \eqref{inf111} comes from a similar argument after evaluating  the next leading terms on the left-hand side in \eqref{inf1}.  This yields a {\em two-dimensional} exponential bundle:
 \be
 \label{inf1N}
  \tex{
 u(r)\approx r^{-\frac{N-1}2} {\mathrm e}^{-r / \sqrt 2} \big[
C_1 \cos\big( \frac {r}{\sqrt 2} \big)+ C_2 \sin\big( \frac
{r}{\sqrt 2} \big)\big]\, ,}%%%% \quad C_{1,2} \in \re, }
 \ee
 where $C_{1}, C_2 \in \re$ are arbitrary parameters of this linearized
 bundle.

%%%%%%%%%%%%%%%%%%%%%%%%%%%%%%%%%%%%%%%%%%%%%%%%%%%%%%

\subsection{Multiplicity results for the BNLSE \eqref{radnls}}

%%%%%%%%%%%%%%%%%%%%%%%%%%%%%%%%%%%%%%%%%%%%%%%%%%%%%

Now,  we perform an analysis that will provide us with an
estimation for the number of stationary solutions or standing waves
for the BNLSE \eqref{bnls22}. To this end, we will first use the
so-called {\em fibering method}, introduced by S.I.~Pohozaev in
the 1970s \cite{Poh0, PohFM}, as a convenient generalization of
previous versions by Clark and Rabinowitz \cite{Clark, Rabin}  of
variational approaches, and further developed   by Dr\'abek and
Pohozaev \cite{DraPoh} and others in the 1980's. This methodology
gives an alternative to other methods such as bifurcation theory
or critical point theory
%%and so on,
in relation  to obtaining multiplicity results of differential equations with a variational structure.

We actually have the uniqueness of the ground states for the BNLSE \eqref{radnls}, although other solutions exist. Indeed, thanks to the
Lusternik-Schnirel'man theory we are able to prove the existence of a countable number of solutions.
However, a different story occurs for the
BNLSE \eqref{bnls22} for which this uniqueness is not true in general and, actually, very difficult to prove.
\vglue 0.2cm
\underline{Fibering Method.} Thus, consider
 the following Euler functional associated to \eqref{radnls} defined by \eqref{funfb}
 such that the solutions of \eqref{radnls} can be obtained as
critical points of the $\mc{C}^1$ functional \eqref{funfb}. To simplify the analysis we write $\mc{F}\equiv \mc{F}_\l$.

Subsequently, we split the function $u\in E$ as follows
\begin{equation}
\label{split}
    \tex{ u(x)=r v(x),}
\end{equation}
 where $r\in \re$, such that $r\geq 0$, and $v\in E$,
 to obtain the so-called {\em fibering maps}
\begin{align*}
    \tex{\phi_v\,:\,} & \tex{ \re \rightarrow \re,}\\  &
    \tex{r \rightarrow \mc{F}(rv).}
\end{align*}
Substituting  $u$ from \ef{split} into  the functional
\eqref{funfb}, we have the following maps
\begin{equation}
\label{funsplit}
    \tex{ \phi_v(r)= \mc{F}(rv)= \frac{r^2}{2} \int\limits_{\re^N} |\Delta v|^2 + \frac{r^2 \l}{2} \int\limits_{\re^N} v^2  -  \frac{r^{4}}{4} \int\limits_{\re^N}
    |v|^{4}.}
\end{equation}
Thus,  \eqref{funsplit} defines the current fibering maps.
 Note that, if $u\in E$ is a critical point of
$\mc{F}(u)$, then
$$
 \tex{
 D_u \mc{F}(rv)v= \frac{\partial \mc{F}(rv)}{\partial r}=0.
  }
$$
In other words, $D_u \mc{F}(rv)v=(D_u \mc{F}(rv) \,| \,v )$. Hence, the calculation of
that derivative yields
$$
    \tex{ \phi_v'(r)=r \Big( \int\limits_{\re^N} |\Delta v|^2 + \l \int\limits_{\re^N} v^2 \Big) -  r^{3} \int\limits_{\re^N}
    |v|^{4},}
$$
where $'=\frac{d}{dr}$.
Moreover, since we are looking for non-trivial
critical points, i.e., $u\neq 0$,  we have to assume that $r\neq
0$.  Also, as usual, the critical
points of the functional $\mc{F}(u)$ in  \eqref{funfb} correspond
to weak solutions of the equation \eqref{radnls}, i.e.,
\begin{equation}
 \label{defweak11}
    \tex{ \int\limits_{\re^N} \Delta u \D \varphi + \l \int\limits_{\re^N} u \varphi  -  \int\limits_{\re^N}
    u^{3}\varphi=0,\quad\hbox{for any} \quad \varphi \in E.}
\end{equation}
Moreover, we also say that $u$ is a critical point when
\begin{equation}
\label{critp}
  \tex{ u\in C:=\{u \in E\,:\,
    D_{u}\mc{F}(u) \varphi=0\quad \hbox{for any} \quad \varphi \in E\}.
 }
\end{equation}
By classic elliptic regularity for higher-order equations
(Schauder's theory; see \cite{Berger} for further details), we
will then always have classical solutions for such equations.

Therefore, using \eqref{split} and looking for $\phi_v'(r)=0$ we actually have
\begin{equation}
\label{varivu}
    \tex{ \int\limits_{\re^N} |\Delta v|^2 + \l \int\limits_{\re^N} v^2 -  r^{2} \int\limits_{\re^N}
    |v|^{4}=0,}
\end{equation}
and assuming that $\int\limits_{\re^N}
    |v|^{4} \neq 0$, we finally arrive at
\begin{equation}
\label{rex}
    \tex{r^{2} = \frac{ \int\limits_{\re^N} |\Delta v|^2 + \l \int\limits_{\re^N} v^2 }{ \int\limits_{\re^N}
    |v|^{4} }\geq 0.}
\end{equation}
Now, calculating $r$ from \eqref{rex} (values of the scalar
functional $r=r(v)$, where those critical points are reached) and
substituting it into \eqref{funsplit} gives the following
functional:
\begin{equation}
\label{spfunc}
    \tex{ \mf{F}(v)=\mc{F}(r(v)v):=\frac 14
    \frac{ \big(\int\limits_{\re^N} |\Delta v|^2 + \l \int\limits_{\re^N} v^2 \big)^{2}}
    {\int\limits_{\re^N} |v|^{4} }\,}.
\end{equation}
Thus, according to Dr\'abek--Pohozaev \cite{DraPoh}, $r=r(v)$ is
well-defined and consequently the fibering map \eqref{funsplit}
possesses a unique point of monotonicity change in the case
\be
\label{poss} \tex{
 \int\limits_{\re^N} |\Delta v|^2 + \l \int\limits_{\re^N} v^2 >0 \quad \hbox{and}\quad
\int\limits_{\re^N} |v|^{4}>0.} \ee
%something that by construction always happens here.
Indeed, thanks to \cite[Lemma\;3.2]{DraPoh}, we have
that the Gateaux derivative of the functional $\mf{F}$ at the
point $v\in E$ in the direction of $v$ is zero, i.e.,
 $$(D_u \mf{F}(v) \,| \,v )=0.
  $$
Therefore, assuming that $v_c$ is a critical point of
$\mf{F}$, by the transformation carried out above, we have that
a critical point $u_c \in E$, $u_c \neq 0$, of
$\mc{F}$ is generated by $v_c$ through the expression
 $$
 u_c=r_c v_c,
$$
 with $r_c$ defined by \eqref{rex}.
 Note that in fact we could have $r>0$, something that can be ascertained just applying the Sobolev embedding \eqref{continuous},
 \eqref{sobo12},
  \be
 \label{posr}
  \tex{ \mf{F}(v)=\frac 14
    \frac{ \big(\int\limits_{\re^N} |\Delta v|^2 + \l \int\limits_{\re^N} v^2 \big)^{2}}
    {\int\limits_{\re^N} |v|^{4} }\geq  C (\l,N)\,}.
\ee
for a positive constant $C$ depending only on $\l>0$, $N$, and related with the corresponding Sobolev's constant of \eqref{continuous},
 \eqref{sobo12}.
Moreover, the different critical points of those fibering maps will
provide us with the critical points of the functional $\mf{F}$
in \eqref{spfunc}, and, hence by construction, of the functional
$\mc{F}$ given by \eqref{funfb}.

Now we state several properties of the fibering maps.
\begin{lemma}
The fibering maps defined by \eqref{funsplit} has a unique point where the  monotonicity change. Precisely, it has a unique critical point which is its global maximum.
\end{lemma}
\begin{proof}
Clearly $\phi_v (r)=c_1 r^2-c_2 r^4$, with $c_1,\: c_2 > 0$; thus $\phi_v(r)$ has a unique point where its monotonicity change and it is where its global máximum is achieved.
In order to be more precise, we use the auxiliary function
\begin{equation}
\label{mult01}
    \tex{\omega_v\,:\,\re_+ \rightarrow \re,\quad  \o_v(r)=  \int\limits_{\re^N} |\Delta v|^2  -  r^{2} \int\limits_{\re^N}
    |v|^{4}.}
\end{equation}
such that $u=r(v)v$ is a critical point for the functional  \eqref{funfb}. Additionally, by construction
for $r>0$ we arrive at the equality
\begin{equation}
\label{imp}
    \tex{\o_v(r)= - \l \int\limits_{\re^N} v^2,}
\end{equation}
and, hence, we have that $rv(r)$ is a critical point of \eqref{funfb}.
Also,
\be\label{eq:derom}
    \tex{ \o_v'(r)= -2r \int\limits_{\re^N}
    |v|^{4},}
\ee
and, since $\int\limits_{\re^N}   |v|^{4}$ is always positive we can conclude that the function $\o_v$ is strictly decreasing for any $r>0$.
Moreover,
\be\label{eq:drom2}
    \tex{ \phi_v''(r)=\int\limits_{\re^N} |\Delta v|^2 + \l \int\limits_{\re^N} v^2  - 3 r^{2} \int\limits_{\re^N}
    |v|^{4}= -2 r^{2}\int\limits_{\re^N}
    |v|^{4} .}
\ee
Then, if $u=r(v)v$ is a critical point for the functional denoted by \eqref{funfb} the following is attained
$$
    \tex{  r^{-1}\o_v'(r)=   \phi_v''(r),}
$$
then, the fibering map $\phi_v$ is concave in a neighborhood  of $r(v)$ and we get a local maximum point, which indeed, by the geometry (described at the beginning of the proof) of $\phi_v$ is a global maximum point.
\end{proof}

\begin{remark}
As we shall prove below, due to the \emph {Lusternik-Schnirel'man analysis} (L--S) the previous result is true for the first critical point. Moreover, using this topological
and variational methodology we are able to prove also that there are a countable number of critical
points for the functionals in hand, with the possible existence of others that are not of minimax type and might not be  identify here.
\end{remark}
%%%%%%%%%%%%%%%%%%%%%%%%%%%%%%%%%%%

\subsection{Spectral properties of the linear associated problem}

%%%%%%%%%%%%%%%%%%%%%%%%%%%%%%%%%%%%

To establish the values of the parameter $\l$ for which we have the existence of one non-trivial solution
we will take into consideration the spectral properties of the following eigenvalue problem
\be
\label{specp} \D^2 \psi_\b=(\ell_{\b,\l}-\l) \psi_\b
\quad\hbox{in}\quad \re^N, \quad \hbox{and} \lim_{|x|\to \infty}
\psi_\b(x)=0, \ee
such that we define the first eigenvalue of the previous problem \eqref{specp} by
\be
\label{raz}
\ell_{1,\l}:=\inf_{\begin{smallmatrix} u\in  H\\ u\not\equiv 0\end{smallmatrix}}  \frac{ \int_{\re^N} |\D u|^2+\l \int_{\re^N} u^2}{\int_{\re^N} u^2}.
\ee
It is obvious that the function $\ell_{1,\l}$ is increasing in $\l$ and positive by the definition of \eqref{raz}.

\vspace{0.2cm}

\begin{remark} We point out that equation \eqref{specp} can be scaled out and, hence, reduced to the eigenvalue
equation
\[\D^2 \psi_\b+ \psi_\b=\ell_\b \psi_\b.\]
However, since it will be important in our subsequent analysis for the system \eqref{s1} we keep the parameter $\l$ in the
following.
\end{remark}
As we have seen before the linearized equation of \eqref{radnls} admits solutions with a proper exponential decay
at infinity, radially symmetric solutions indeed.

Furthermore, the eigenvalues of the other problem (but similar)
\be
\label{exspecp} \D^2 \psi_\b=\hat{\ell}_{\b} \psi_\b
\quad\hbox{in}\quad \re^N, \quad \hbox{and} \lim_{|x|\to \infty}
\psi_\b(x)=0 \ee
are strictly positive. Basically thanks to the positivity of the left hand side of \eqref{exspecp}, with the first eigenvalue
satisfying the expression
 \be
\label{rai}
\hat\ell_1:=\inf_{\begin{smallmatrix} u\in  H\\ u\not\equiv 0\end{smallmatrix}}  \frac{ \int_{\re^N} |\D u|^2}{\int_{\re^N} u^2}.
\ee
Also, for  $\alpha$ large enough the resolvent operator $(\D^2 +\alpha {\rm Id})^{-1}$ corresponding to the eigenvalue problem \eqref{exspecp} (and
equivalently \eqref{specp} with $\l$) is compact and, hence, there exists a discrete family of isolated real eigenvalues
\[ 0<\hat\ell_1\leq \hat\ell_2\leq \ldots \leq \hat\ell_\b \leq \ldots\]
Similarly, assuming
\be
\label{bigeq}
0<\l <\ell_{1,\l},
\ee
we also find   a discrete family of isolated real eigenvalues for the eigenvalue problem \eqref{specp}
\[ 0<\ell_{1,\l}\leq \ell_{2,\l}\leq \ldots \leq \ell_{\b,\l} \leq \ldots\, \hbox{ so that }\, \ell_{1,\l}:=\inf_{\begin{smallmatrix} u\in
H\\ u\not \equiv 0\end{smallmatrix}}  \frac{ \int_{\re^N} |\D u|^2}{\int_{\re^N} u^2}+\l=\hat\ell_1+\l>0,\]
then since $\hat\ell_1>0$ we actually have that $\l<\ell_{1,\l}$.
We also point out that those families of eigenvalues in fact tend to infinity.

Under those assumptions we have the following
variational expression of the problem \eqref{specp} $$\tex{
\int_{\re^N} \D u \D v =(\ell_\l-\l) \int_{\re^N} uv,\quad \hbox{for
any}\quad v\in H.}$$
Thus, $u\in H\setminus\{0\}$ is an eigenfunction of the
problem \eqref{specp} associated to the eigenvalue $\ell_\l$, that will depend on the value of the parameter $\l$. Moreover,
to have that weak formulation of the problem we need the following
result.
\begin{lemma}
\label{leexpin} Suppose that $u\in H$. Then,
there is a sequence $\{R_n\}\subset \re_+$, with $R_n\to \infty$
as $n\to \infty$, such that $$ \lim_{n\to \infty}  \int_{\p
B_{R_n}} \nabla u \frac{\p \nabla u}{\p {\bf n}} dS=0 \quad
\hbox{and}\quad \lim_{n\to \infty}  \int_{\p B_{R_n}}  u \frac{\p
\D u}{\p {\bf n}} dS=0,$$ where $B_{R_n}$ is the ball of radius
$R_n$ and centered at the origin, and ${\bf n}=\frac{x}{\|x\|}$,
with $\|x\|=R_n$ is the unitary outward normal vector.
\end{lemma}

Consequently we can obtain certain monotony results for those fibering maps depending on the value of the parameter $\l$.
\begin{lemma}
There exists $r_{\rm max}>0$ such that
the fibering maps $\phi_v(r)$ are increasing for $r<r_{\rm max}$ and decreasing for $r>r_{\rm max}$ provided $\l\in \left(-\frac{1}{K},\infty\right)$, with
 $$\tex{K=\frac{2}{\hat\ell_1},}$$
 where $\hat\ell_1$ is the first eigenvalue of problem \eqref{exspecp}.
 \end{lemma}
 \begin{proof}
According to the definition of the fibering map denoted by \eqref{funsplit} if the parameter $\l$ is very small, i.e., $\l\ll 0$, the fibering
map is a decreasing function and without critical points. Moreover, we observe that if the parameter $\l$ is bigger than a certain value
(to be ascertained below) the fibering map is positive, $\phi_v(r)>0$, since $\int\limits_{\re^N}
    |v|^{4}>0$ at least for sufficiently small $r$'s.
Then, considering the functional
$$
    \tex{\mc{H}_v(r):=\frac{r^2}{2} \int\limits_{\re^N} |\D v|^2 - \frac{r^{4}}{4} \int\limits_{\re^N}
    |v|^{4},}
$$
we can assure that
$\mc{H}_v(r)$ has a unique critical point at
the value $r=r_{\max}$ such that
\be
\label{maxr}
    r_{\max}= \left(\frac{ \int\limits_{\re^N} |\D v|^2}{\int\limits_{\re^N}
    |v|^{4}}\right)^{\frac{1}{2}}   \quad \hbox{and}\quad \mc{H}_v(r_{\max})= \frac 14\frac{\left(\int\limits_{\re^N} |\D v|^2\right)^{2}}{\int_{\re^N}
    |v|^{4}}.
\ee
Note that $\mc{H}_v(r)$ is clearly increasing in the interval
$(0,r_{\max})$, for sufficiently small $r$'s. Subsequently, by the Sobolev imbedding \eqref{continuous}, \eqref{sobo12} we have that
$$
    \tex{\mc{H}_v(r_{\max})\geq C_N,}
$$
where $C_N>0$ is related to  the Sobolev's constant. Finally, we
will prove that there exists a value of the parameter $\l=\l^*$ such that $\phi_v(r_{\max})>0$, i.e.,
$$\tex{\mc{H}_v(r_{\max}) + \frac{r_{\max}^2 }{2} \l \int\limits_{\re^N} v^2  >0,}$$
for any $u \in W^{2,2}(\re^N) \setminus\{\bo\}$ and $\l >\l^*$. Indeed, thanks to the first eigenvalue in problem \eqref{exspecp} and \eqref{maxr}
$$
    \tex{\frac{r_{\max} ^2 }{2} \int\limits_{\re^N} u^2 \leq \frac{1}{2\hat\ell_1} \frac{ \int\limits_{\re^N} |\D v|^2}{\int\limits_{\re^N}
    |v|^{4}}  \int\limits_{\re^N} |\D v|^2 =  \frac{2}{\hat\ell_1}\mc{H}_v(r_{\max}) .}
$$
Then, if
$$\tex{\hbox{if}\quad K= \frac{2}{\hat\ell_1} \quad \hbox{we find that}\quad \frac{r_{\max} ^2 }{2} \int\limits_{\re^N} u^2 \leq K \mc{H}_v(r_{\max}) },$$
and, hence
$$
    \tex{\phi_v(r_{\max})\geq \mc{H}_v(r_{\max}) + \l K \mc{H}_v(r_{\max})= \mc{H}_v(r_{\max}) (1+\l K).}
$$
Therefore, the fibering map denoted by \eqref{funsplit} is defined positive at $r_{\rm max}$ for all non-zero $u$ if $\l > -\frac{1}{K}$.

 \end{proof}

 \begin{remark}
Note that we are assuming that $\l>0$ so that for the existence of solutions we just consider positive values of the parameter.

 This fact is not very important to get the first solution through these fibering map techniques but it will be crucial in
 ascertaining the countable family of minimax solutions via Lusternik-Schnirel'man analysis; see below.
\end{remark}

\begin{remark} Furthermore, if we assume that
 \be
 \label{condst}
 \l > -\frac{1}{K}\quad \hbox{and}\quad \ell_1(\l)-\l>0,
 \ee
 it is clear that, by the fibering method, there exists exactly one solution of
\eqref{imp}. Indeed, the fibering map $\phi_v(r)$ is a strictly increasing
function for $r<r_{\max}$, and decreasing for $r>r_{\max}$. Thus,
there exists a unique value of
$$
r_1(v)>0\quad \hbox{such
that}\quad r_1(v)v=u
  $$
  is a critical point of the functional
$\mc{F}(u)$ in \eqref{funfb}. Also, thanks to \eqref{eq:derom} we get
$\o_\l'(r_1(v))<0$
the unique critical point $r=r_1(v)$, that fibering map $\phi_v$
has a local maximum, since $\phi_v''(r_1(v))<0$ due to \eqref{eq:drom2}.
\end{remark}

\noindent\underline{Lusternik-Schnirel'man analysis}.
Through the topological method due to Lusternik-Schnirel'man
we are able to establish the existence of multiple solutions for the functional $\mc{F}$ \eqref{funfb}.
This theory is based on determining a topological analogue for the minimax principles which characterize the eigenvalues of self-adjoint
compact operators $L$. If $\ell_1,\ell_2,\cdots$ denote the real
eigenvalues of a self-adjoint compact operator $L$, ordered by
their values with multiplicities it yields
$$
  \tex{\ell_\b=\sup\limits_{[S^{N-1}]}\,\,\, \min\limits_{u\in S^{N-1}} (Lu \,| \,u ),}
$$
 where $S^{N-1}$ denotes the unit sphere in
an arbitrary $N$-dimensional linear subspace $\Sigma$ of the
corresponding functional space $H$, and $[S^{N-1}]$ denotes the
class of such spheres as $\Sigma$ varies in $H$. Thus, applying
the calculus of variations theory to an operator $L$, the
eigenvalues of the operator $L$ are precisely the critical values
of the functional $( Lu\, |\, u)$ on the unit
ball $\p\Sigma=\{u\,:\,\|u\|=1\}$ of $H$.

To extend it to general smooth functionals such as $\mc{F}$, we just need to find the topologic analogues for the sets $S^{N-1}$.
Indeed, in our particular case this functional subset is the following
\begin{equation}
 \label{R0}
    \tex{ \mc{R}_{0,\l}=\Big\{u\in H\,:\, \int\limits_{\re^N} |\D u|^2 +\l
    \int\limits_{\re^N} u^2 =1\Big\}.}
\end{equation}
\begin{remark} Note that we could have performed the previous \emph{fibering analysis}
restricted to this subset \eqref{R0} yielding similar results. Indeed, following the L--S theory
through a \emph{fibering analysis} approach the number of critical points of the functional \eqref{funfb} associated with
the elliptic equation \eqref{radnls} depends on the \emph{category} (or \emph{genus}) of the functional subset \eqref{R0}
on which the fibering is taking place.
\end{remark}

\vspace{0.2cm}

Consequently, the critical values\footnote{Clearly, one can prove without difficulties the Palais-Smale condition in order to achieve these critical values, since we are working on compact subsets of the radial space $H$; see also Remark\;\ref{rem:compact}.}  $c_\b$ and the corresponding
critical points $\{u_\b\}$ are:
\begin{equation}
\label{cat} \tex{c_\b := \inf\limits_{\mc{A}\in \mc{A}_\b}
\sup\limits_{u\in \mc{A}} \mc{F} (u)} \quad  (\b=1,2,3,...),
\end{equation}
and
\be
\label{setb}
  \mc{A}_\b :=\{\mc{A}\,:\, \mc{A}\subset
\mc{R}_{0,\l},\,\hbox{compact subsets},\quad \mc{A}=-\mc{A}\quad
\hbox{and}\quad \g(\mc{A}) \geq \b\},
  \ee
   is the class of closed
sets in $\mc{R}_{0,\l}$ such that, each member of $\mc{A}_\b$ is of
genus (or category) at least $\b$ in $\mc{R}_{0,\l}$ and invariant under any odd continuous amp. The fact that
$\mc{A}=-\mc{A}$ comes from the definition of genus
(Krasnosel'skii \cite[p.~358]{Kras}) such that, if we denote by
$\mc{A}^*$ the set disposed symmetrically to the set $\mc{A}$, $$
  \mc{A}^*=\{ u\,:\, u^*=-u\in \mc{A}\},
  $$
   then, $\g(\mc{A})=1$
when each simply connected component of the set $\mc{A} \cup
\mc{A}^*$ contains neither of the pair of symmetric points $u$ and
$-u$.

Then to obtain the critical points of a functional on
the corresponding functional subset $\mc{R}_{0,\l}$, one needs to
estimate the category $\g$ of that functional subset. Thus, the
category will provide us with the number of critical points that
belong to the subset $\mc{R}_{0,\l}$. Namely, similar to \cite{PohFM, GMPSob}
we state the following.
\begin{lemma}
The category $\g(\mc{R}_{0,\l})$ is given by the number of eigenvalues
(with multiplicities) of the corresponding linear eigenvalue
problem satisfying:
 \be
 \label{rho1}
  \tex{
\g(\mc{R}_{0,\l})= \sharp \{\ell_{\b,\l} >0\}, \quad \mbox{where } \ell_{\b,\l} \mbox{ is defined by } \eqref{specp}}
\ee
 \end{lemma}
 \begin{proof}
 Let $\ell_{\b,\l}$ be the $\b$-eigenvalue of the linear bi-harmonic problem \eqref{specp} such that
 \be
 \label{psi11}
  \psi_\b:=\sum_{k\geq 1} a_k
\hat{\psi}_k,
  \ee
taking into consideration the multiplicity of the $\b$-eigenvalue, under the natural ``normalizing"
constraint $$\sum_{k\geq 1} a_k=1.
  $$
   Here, \ef{psi11} represents
the associated eigenfunctions to the eigenvalue $\ell_{\b,\l}$ and
$$\{\hat{\psi}_1,\cdots,\hat{\psi}_{M_\b}\},$$
is a basis of the
eigenspace of dimension $M_\b$.
Moreover, assume a critical point
\be
\label{u88}
u=\sum_{k\geq 1} a_k
\hat{\psi}_k,
\ee
belonging to the functional subset \eqref{R0}  and the eigenspace of dimension $M_\b$.

 Thus, substituting \eqref{u88} (since we are looking for solutions of that form) into the equation \eqref{radnls} with $\s=1$, and using the expression of the spectral problem \eqref{specp} yields
 $$\tex{\sum_{k\geq 1} a_k \ell_k \hat{\psi}_k - \left(\sum_{k\geq 1} a_k \hat{\psi}_k\right)^3=0,}
 $$
 which provides us with an implicit condition for the coefficients $a_k$ corresponding to the critical point \eqref{u88}. Indeed, assuming normalized
 eigenfunctions $\psi_\b$, i.e.,
 $$\int\psi_\b^2=1,$$
  and multiplying by $u$ in \eqref{radnls} and
 integrating we have that
 $$\tex{\sum_{k\geq 1} a_k^2 \ell_k- \int_{\re^N} \left(\sum_{k\geq 1} a_k \hat{\psi}_k\right)^4=0.}$$
   Furthermore, taking into account that $u\in \mc{R}_{0,\l}$  one arrives at the relation
 \[\sum_{k} a_k^2 \ell_k=1.\]
 Therefore, in order to have the sphere \eqref{R0} we actually find that the category of that subset $\mc{R}_{0,\l}$
 must satisfy the expression \eqref{rho1}.
 \end{proof}
Subsequently, since
functional \eqref{funfb} is symmetric
we can state the following result with which we establish that the elliptic problem \eqref{radnls}
possesses a countable set of different solutions obtained as critical points of the functional  $\mc{F}$
denoted by \eqref{funfb}.
\begin{theorem}
\label{th21}
$\mc{F}$ possesses an unbounded sequence (family) of critical points.
\end{theorem}

\begin{proof}
We estimate the category by approximation. Namely, as
customary, we consider our equation in a ball
of arbitrarily large radius $R>0$ with homogeneous Dirichlet
boundary conditions:
 \be
 \label{D1}
 \D^2_r \psi_\b= (\ell_{\b,\l}(R)-\l) \psi_\b \quad \mbox{in} \quad B_R.
  \ee
In the radial geometry, by obvious  scaling we deduce the spectral
problem to that in $B_1$:
 \be
 \label{D2}
 R^{\frac 14} r =y \LongA \D^2_y \psi_\b= (\ell_{\b,\l}(R)-\l) R^4 \psi_\b \quad \mbox{in} \quad B_1.
 \ee
 Denoting by $\{\a_k>0\}$ the spectrum of this self-adjoint
 operator in $B_1$, we see a simple dependence
  \[
  \tex{
 \ell_k(R)-\l = \frac {\a_k}{R^4} \to 0 \asA R\to \iy.
 }
 \]
 Then clearly
  \[
  \tex{
 \ell_{\b,\l}(R) \to \l \asA R \to \iy,
 }
 \]
because $(\ell_{\b,\l}(R)-\l) R^4$ is a positive constant independent of $R$, thanks to \eqref{bigeq}. Moreover, due to \ef{rho1},
applied to the approximating problem \ef{D1}, we
then conclude that the category (as the number of critical points) gets arbitrarily large as $R
\to \iy$ which ensures that the category of the set \ef{R0} in
$\ren$ is infinite.
\end{proof}
\begin{remark}
The
Lusternik-Schnirel'man theory  cannot assure that there exists
a precise number of solutions since this topological method provides us with a
countable family of solutions and
the possibility of having more than that number of
solutions cannot be ruled out.

Moreover,  from the analysis performed above if the
parameter $\l$ is sufficiently small, i.e., $\l<  -\frac{1}{K}$,
and assuming only positive solutions, there are no such critical
points, since the fibering map $\phi_v$ is then  a strictly
decreasing function.

However, assuming oscillatory solutions of changing sign, we
show below that the number of  possible critical points of the
functional \eqref{funfb} increases while decreasing the value of
the parameter $\l$. Indeed, fix a value of the parameter $\l$
bigger than $-\frac{1}{K}$ and so that the following condition is satisfied
 \be
 \label{conell}
 \tex{
\ell_{\b,\l} - \l >0,
 }
\ee
 where
$\ell_{\b,\l}$ is the $\b$-eigenvalue of the linear bi-harmonic operator
\eqref{specp} for the eigenfunction  \eqref{psi11}.

Thus, we
obtain that, for a solution of the form $u=r \psi_\b$, we will
have $M_\b$ corresponding solutions similar to the one obtained in
the previous case, i.e., when the parameter $\l >- \frac{1}{K}$. Indeed, substituting $u=r \psi_\b$ into
the functional \eqref{funfb} for $\l=\ell_{\b,\l}-\e$ (so that condition \eqref{conell} is fulfilled), then we have
 $$
\mc{F}(r\psi_\b):= - \frac{r^2 \e}{2} \int\limits_{\re^N} v^2  -  \frac{r^{4}}{4} \int\limits_{\re^N}
    |v|^{4},
    $$
  and performing a similar analysis as the one done previously,
   we will find $\b$--critical points (corresponding to the dimension of the eigenspace).
\end{remark}
%%%%%%%%%%%%%%%%%%%%%%%%%%%%%%%%%%%%%%%%%%%%%%%%%%%%%%

\section{Existence results}\label{Sec3}

%%%%%%%%%%%%%%%%%%%%%%%%%%%%%%%%%%%%%%%%%%%%%%%%%%%%%
\noindent In this section we will mainly prove that the infimum of $\mathcal{J}$ constrained on $\cN$ is achieved.

 \begin{remark}\label{rem:ground-eq}
 If we denote by $U$ the radially symmetric ground state of the equation $\D^2 u +  u = u^3,$ then
 $$
 \tex{
 U_j(x)=\sqrt{\frac{\l_j}{\mu_j}}U(\l_j^{1/4}x)
 }
 $$ is the radially symmetric ground state solution to
$$\D^2 u +  \l_ju = \mu_ju^3.$$ As a consequence,
for any $\b\in\re$, system \eqref{s1} possesses two semi-trivial solutions
$$\bu_1=(U_1,0), \quad \bu_2=(0,U_2).$$
Moreover,
proving the existence of nontrivial solutions of \eqref{s1}
requires proving that both components are not zero, i.e., the new solutions must be different from $\bu_j$,
with $j=1,\, 2$.
\end{remark}
We define the following Sobolev constants associated to the weights $U_j$ previously defined.
$$
S_1^2  =  \dyle\inf_{\varphi\in E\setminus\{0\}}
\frac{\|\varphi\|_2^2}{\int_{\re^N} U_1^2\varphi^2},\qquad
S_2^2  =  \dyle\inf_{\varphi\in E\setminus\{0\}}
\frac{\|\varphi\|_1^2}{\int_{\re^N} U_2^2\varphi^2},
$$
and
$$
\Lambda = \min \{S_1^2,S_2^2\},\qquad \Lambda'=\max \{S_1^2,S_2^2\}.
$$
The first step in order to prove existence results, as in \cite{ac2}, consists of proving the following.
%%%%%%%%%%%%%%%%%%
\begin{proposition}\label{pr:ac1}
$(i)$  $\forall\,\b<\Lambda$, the semi-trivial solutions $\bu_j$, $j=1,2$, are strict local
minima of $\mathcal{J}$ constrained on $\cN$.

$(ii)$  If $\b>\Lambda'$ then both $\bu_j$ are saddle points of $\mathcal{J}$ on $\cN$. In particular,
$$\inf_{\cN}\mathcal{J} < \min\{\mathcal{J}(\bu_1),\mathcal{J}(\bu_2)\}.$$
\end{proposition}
The previous result follows easily from the proof of  Proposition 4.1 in \cite{ac2}   with the appropriate changes,
but we include it for the sake of completeness.
The idea consists on the evaluation of the Morse index of $\bu_j$, as critical points of $\mathcal{J}$ constrained on $\cN$.
To this end, let us denote $D^2\mathcal{J}_{\cN}$ as the second derivative of $\mathcal{J}$ constrained on $\cN$. Since
$\mathcal{J}'(\bu_j)=0$, then
$$D^2\mathcal{J}_{\cN} (\bu_j)[\bh]^2=\mathcal{J}'' (\bu_j)[\bh]^2\quad \hbox{for any}\quad \bh\in T_{\bu_j}\cN,$$
where $T_{\bu_j}\cN$ denotes the tangent space to $\cN$ at $\bu_j$.
Moreover, if we define the Nehari manifold associated to $I_j$ by
$$
 \tex{
\cN_j =\{u\in E : (I_j'(u)|u)_j=0\}=\Big\{u\in E : \|u\|_j^2
-\mu_j \int_{\re^N} u^{4}=0\Big\},
 }
 $$ since $I'_j(U_j)=0$, then $$D^2 (I_j)_{\cN_j} (U_j)[h]^2=I_j''
(\bu_j)[h]^2\quad \hbox{for any}\quad h\in T_{U_j}\cN_j.$$
According to \cite{ac2}, it is easy to prove that $$\bh\in
T_{\bu_j}\cN\quad \hbox{if and only if} \quad h_j\in
T_{U_j}\cN_j,\quad j=1,\, 2.$$

\

\noindent {\it Proof of Proposition \ref{pr:ac1}.}
 $(i)$ Note that
$$
 \tex{
 \mathcal{J}''(\bu_1)[\bh ]^2 =I_1''(U_1)[ h_1]^2 + \| h_2\|_2^2
-\b \int_{\re^N} U_1^2 h_2^2,\qquad \forall\, \bh\in T_{\bu_1}\cN.
 }
$$
 Since $U_1$ is a minimum of $I_1$ on $\cN_1$ there exists
$c_1>0$ such that \be\label{eq:minimum} I_1''(U_1)[h]^2 \geq c_1
\|h\|_1^2,\qquad \forall \,h\in T_{U_1} \cN_1. \ee
   Taking $\bh\in
T_{\bu_1}\cN$, i.e., $h_1\in T_{U_1}\cN_1$ and using
\eqref{eq:minimum} we get
\begin{eqnarray*}
\mathcal{J}''(\bu_1)[\bh]^2&\geq & c_1
\| h_1\|_1^2+\|h_2\|_2^2
 \tex{
-\b \int_{\re^N}
 } U_1^2 h_2^2\\ &\geq&c_1 \| h_1\|_1^2+\|
h_2\|_2^2
 \tex{
 -\frac{\b}{S_1^2}\| h_2\|_2^2.
 }
\end{eqnarray*}
Therefore, if $\b<S_1^2$ there exists $c_2>0$ such that
$$
\mathcal{J}''(\bu_1)[\bh]^2\geq c_1 \| h_1\|_1^2+c_2\| h_2\|_2^2.
$$
Similarly,  if $\b<S_2^2$,  $\exists\;c'_j>0$, $j=1,\, 2$ such that
$$
\mathcal{J}''(\bu_2)[\bh]^2\geq c'_1 \| h_1\|_1^2+c'_2\| h_2\|_2^2.
$$

$(ii)$ Assume $\b>S_1^2$, then there exists $\overline{\psi}\in E$ such that
$$
 \tex{
S_1^2 < \frac{\|\overline{\psi}\|_2^2}{\int_{\re^N} U_1^2\overline{\psi}^2}<\b.
 }
$$
 In particular,  one has that $(0,\overline{\psi})\in
T_{\bu_1}\cN$. Therefore
  $$
   \tex{
\mathcal{J}''(\bu_1)[(0,\overline{\psi})]^2 =
\|\overline{\psi}\|_2^2 -\b \int_{\re^N}U_1^2\overline{\psi}^2<0.
 }
$$
 Similarly, if $\b>S_2^2$, there exists $(\overline{\phi},0)\in T_{\bu_2}\cN$
such that $\mathcal{J}''(\bu_2)[(\overline{\phi},0)]^2<0$.
\rule{2mm}{2mm}

\begin{remark}\label{rem:a1}
Note that Proposition \ref{pr:ac1} can be read as $\bu_j$ are strict local  minima, resp. saddle points, provided $\b<S_j^2$, resp. $\b>S_j^2$, $j=1,2$.
\end{remark}
Concerning  the PS condition (see Remark\;\ref{rem:PS}-$(2)$) we prove the following result.
\begin{lemma}\label{lem:PS}
Assume that $2\le N<8$, then $\mathcal{J}$ satisfies the PS condition constrained on $\cN$.
\end{lemma}
\begin{pf}
Let $\bu_n\in \cN$ be a sequence such that $\mathcal{J}(\bu_n)\to c>0$, as $n\to \infty$.
From \eqref{eq:M3}  it follows that
$\bu_n$ is bounded and,  without relabeling, we can assume that $\bu_n\rightharpoonup \bu_0$.
 Since $H$ is compactly embedded into $L^4(\re^N)$ for $2\le N<8$ (see Remark\;\ref{rem:compact}), we infer that
 $$F(\bu_n)+\b G(\bu_n)\to F(\bu_0)+  \b G(\bu_0).$$
Moreover using \eqref{eq:F} jointly with \eqref{eq:M1}, one has that
$$\exists\,c>0\quad \hbox{such that}\quad F(\bu_n)+\b G(\bu_n)\geq c,\quad \hbox{and then}\quad \bu_0\not= 0.$$
Letting
  $$\n_{\cN}\mathcal{J}(\bu)=\mathcal{J}'(\bu)-\o \Psi'(\bu),$$   denote the constrained gradient of $\mathcal{J}$ on $\cN$, with $\o\in \re$. Suppose that
 $$\n_{\cN}\mathcal{J}(\bu_n)\to 0.$$ Taking the scalar product with $\bu_n$ and recalling that
 $(\mathcal{J}'(\bu_n)\mid \bu_n)=\Psi(\bu_n)=0$, we find that
 $$\o_n (\Psi'(\bu_n)\mid \bu_n)\to 0,$$
 and this,
jointly with \eqref{eq:M2}, implies  that $\o_n\to 0$. Moreover, taking into account that
$$\|\Psi'(\bu_n)\|\leq c_1<\infty,$$
we deduce that $\mathcal{J}'(\bu_n)\to 0$. To finish the proof, since $\lim (\mathcal{J}'(\bu_n)\mid \bu_0)= 0$ one can conclude that  $\bu_n\to \bu_0$ strongly.
\end{pf}

Note that Proposition \ref{pr:ac1} and Lemma \ref{lem:PS}
will be useful in the proof of the main result dealing with the existence of non-trivial solutions different from the semi-trivial solutions.
\begin{theorem}\label{th:main}
$(i)$ If   $\b<\Lambda$, then $\mathcal{J}$ has a Mountain-Pass  (MP) critical point  $\bu^*$ on $\cN$, and there holds
$\mathcal{J}(\bu^*)>\max\{\mathcal{J}(\bu_1) ,\mathcal{J}(\bu_2)\}$.

\noindent $(ii)$ If $\b>\Lambda'$ then $\mathcal{J}$ has a global minimum $\widetilde{\bu}$ on $\cN$, and there holds
$\mathcal{J}(\widetilde{\bu})<\min\{\mathcal{J}(\bu_1) ,\mathcal{J}(\bu_2)\}$.
\end{theorem}
\begin{pf}
$(i)$  Due to Proposition \ref{pr:ac1}-$(i)$, $\bu_j$ ($j=1,\, 2$)  are strict local minima of $\mathcal{J}$ on $\cN$. This fact allows us to
apply the Mountain Pass Theorem (MPT for short, see \cite{ar}) to $\mathcal{J}$ on $\cN$, yielding a PS sequence $\{\bv_n\}\subset\cN$ with
$$\mathcal{J} (\bv_n)\to c,\quad \hbox{where}\quad
c=\inf_{\Gamma}\max_{0\le t\le 1}\Phi (\g (t)),
$$
and $$\G=\{ \g:[0,1]\to\cN\mbox{ continuous }|\: \g(0)=\bu_1,\: \g(1)=\bu_2,\}.$$
By Lemma \ref{lem:PS}, we find a convergent subsequence (if necessary without relabeling)
$$\bv_n\to \bu^*,\quad \hbox{strongly in} \quad \mathbb{H},$$
so that  $\bu^*$ is a critical point of $\mathcal{J}$ and, hence, thanks to Proposition \ref{pr:ac} $\bu^*\in\cN$.
Moreover, by  the MPT again, it also follows that
$$\mathcal{J}(\bu^*)>\max\{\mathcal{J}(\bu_1) ,\mathcal{J}(\bu_2)\}.$$

$(ii)$ Now, due to Lemma \ref{lem:PS}, the $\inf_{\cN}\mathcal{J}$ is achieved at some $\widetilde{\bu}\in\cN$.
Moreover, if $\b>\Lambda'$,  by Proposition \ref{pr:ac1}-$(ii)$ we get that
$$\mathcal{J}(\widetilde{\bu})<\min\{\mathcal{J}(\bu_1) ,\mathcal{J}(\bu_2)\}.$$
Note that if $\widetilde{\bu}$ had for example, the second component zero, then clearly
$\widetilde{\bu}=(\widetilde{u}_1,0)$ with $\widetilde{u}_1\not=0$, but in that case, $\widetilde{u}_1$
would be a non-trivial solution of the equation $\D^2 u +  \l_1u = \mu_1 u^3$ with energy strictly less
than the energy of $u_1$ which is a ground state of the previous equation (see Remark \ref{rem:ground-eq}),
and this is a contradiction. Arguing in a similar way we conclude that the first component of
$\widetilde{\bu}$ is non-trivial too.
\end{pf}
\begin{remarks}
\begin{enumerate}\label{rem:a2}
\item Note that statement $(i)$ of  Theorem \ref{th:main} is weaker than one could expect, since although $\bu^*\not =\bu_j$, $j=1,2$,
it does not exclude that $\bu^*$ might be a solution of \eqref{s1} with one component zero. This is not the case in statement $(ii)$ of Theorem \ref{th:main} as we have proved.
\item We observe as in \cite{ac2}, in order to prove the preceding theorem, it would be enough that only one among $\bu_j$
is a minimum or a saddle. For example, if
$$\mathcal{J}(\bu_1)<\mathcal{J}(\bu_2),$$
to prove $(i)$
 it suffices that $\bu_2$ is a minimum. According to Remark \ref{rem:a1}, this is the case
 provided $\b<S_2^2$. Unfortunately,  a straight calculation shows that
$$ \mbox{if } \mathcal{J}(\bu_1)<\mathcal{J}(\bu_2) \mbox{ then } S_2^2<S_1^2.$$
Hence $\bu_1$ is a minimum as well. The same remark holds for the case $(ii)$.
\end{enumerate}\end{remarks}
%%%%%%%%%%%%%%%%%%%%%%%%%%%%%%%%%%%%%%%%%%%%%%%%%%%%%%

\section{Multiplicity results for the system \eqref{s1}}\label{Sec4}

%%%%%%%%%%%%%%%%%%%%%%%%%%%%%%%%%%%%%%%%%%%%%%%%%%%%%
We analyse the multiplicity results in relation to the bi-harmonic
nonlinear Schr\"{o}dinger system \eqref{s1}. This analysis could
provide us with some answers to the discussion made at the end of
last section in which we did not exclude the possibility of having
$\bu^*$ as a solution of \eqref{s1} and with one nil component.

To this end we perform an analysis that will provide us with an estimation in the number of solutions for \eqref{s1}.
Thus, we will first show an application of the so-called {\em fibering method} used in Section\;\ref{Sec2}
for the one single fourth order Schr\"{o}dinger equation \eqref{radnls}.

\vspace{0.2cm}

\noindent\underline{Fibering Method.} Consider the following Euler
functional associated with \eqref{s1}:
$$
    \tex{ \mathcal{J}_\l(\bu)=\mathcal{J}_\l(u_{1},u_{2})= I_1(u_{1})+I_2(u_{2}) - \b\, G(u_{1},u_{2}),
    }
$$
defined by \eqref{mainfun0} such that the solutions of \eqref{s1} can be obtained as
critical  points of the $\mc{C}^1$ functional \eqref{mainfun0}. To simplify the analysis we again write $\mc{J}\equiv \mc{J}_{\l}$.

Subsequently, we split the functions $u_1,u_2\in W^{2,2}(\re^N)$ as follows
\begin{equation}
\label{split2}
    \tex{ u_1(x)=r v_1(x),\quad u_2(x)=r v_1(x),}
\end{equation}
 where $r\in \re$, such that $r\geq 0$ (we shall shortly see  this) and $\bv=(v_1,v_2)\in \mathbb{H}$,
 to obtain the so-called {\em fibering maps}
\begin{align*}
    \tex{\Phi_{\bv}\,:\,} & \tex{ \re \rightarrow \re,}\\  &
    \tex{r \rightarrow \mc{J}(r\bv).}
\end{align*}
Then we get
\begin{equation}
\label{funsplit2}
\Phi_{\bv}(r) = \mc{J}(r\bv)= \frac{r^2}{2} \|\bv\|^2 -r^4 F(\bv)-r^4\b G(\bv).
\end{equation}
Thus,  \eqref{funsplit2} defines the current fibering maps.
 Note that, if $\bu =(u_1,u_2)\in \mathbb{H}$ is a critical point of
$\mc{J}(\bu)$, then thanks to \eqref{split2},
\be
\label{mapfib}
 \tex{
 D_{\bu} \mc{J}(r \bv)\bv= \frac{\partial \mc{J}(r \bv)}{\partial r}=0,\quad \hbox{i.e.}\quad  \frac{\partial \Phi_{\bv}(r) }{\partial r}=0.
  }
\ee
In other words, $D_{\bu} \mc{J}(r \bv)\bv=(D_{\bu}\mc{J}(r \bv)\mid\bv )=0$. Hence, the calculation of
that derivative yields
$$
 \Phi_{\bv}'(r) = r\| \bv\|^2 -r^3 (4F(\bv)+4\b G(\bv)).
 $$

Moreover, since we are looking for non-trivial solutions
(critical points), i.e., $\bu\neq (0,0)$, with at least one of the components different from zero,  we have to assume that $r\neq
0$.
Therefore, since we are looking for $r\neq 0$ such that  $\Phi_{\bv}'(r)=0$ and according to \eqref{mapfib} we actually have
\begin{equation}
\label{varivu2}
\| \bv\|^2 -r^2 (4F(\bv)+4\b G(\bv))=0,
\end{equation}
and, in order to have non-trivial solutions (for a certain $\b$ to be shown below)
$$F(\bv)+\b G(\bv)\neq 0,$$
hence,
\begin{equation}
\label{rex2}
   r^2= \frac{\| \bv\|^2}{4F(\bv)+4\b G(\bv)}>0.
\end{equation}
Now, calculating $r$ from \eqref{rex2} (values of the scalar
functional $r=r(\bv)$, where those critical points are reached) and
substituting it into \eqref{funsplit2} gives the following
functional:
\begin{equation}
\label{spfunc2}
    \mc{J}(r(\bv)\bv)=\frac{1}{16} \frac{\|\bv\|^4}{F(\bv)+\b G(\bv)}.
\end{equation}
Note that applying a similar argument to the one performed in \eqref{posr} through the use of the Sobolev's
embedding \eqref{continuous}, \eqref{sobo12} we actually have that $r>0$ and the positivity of the fibering maps as well, i.e.
$$\mc{J}(r(\bv)\bv) \ge C,\quad  \mbox{for some positive constant $C=C(\l_j,\mu_j,N)$.}$$
 In fact we already had that since previously, by \eqref{eq:bound}, we obtained that
$$\tex{\mc{J}{|_{_{\mc{N}}}} >C>0.}$$
Hence, this result means that the fibering maps never cut through the axis although, and as we shall see
below, due to the Lusternik-Schnirel'man analysis
they can have infinitely many critical points.

Thus, we have the following result.

\begin{lemma}
\label{rwell}
$r=r(\bv)$ is
well-defined and consequently the fibering map \eqref{funsplit2}
possesses a unique point of monotonicity change in the case
\be
\label{poss2}
 4(F(\bv)+\b G(\bv))=\tex{ \mu_1 \int_{\mathbb{R}^N} v_1^{4}\,{\mathrm d} x +
    \mu_2 \int_{\mathbb{R}^N} v_2^{4}\,{\mathrm d} x +    2\beta \int_{\mathbb{R}^N} |v_1|^2|v_2|^2 \,{\mathrm d} x>0.}
    \ee
such that $\l_j>0$, with $j=1,2$, and $\b >-\sqrt{\mu_1\mu_2}$.
\end{lemma}

\begin{proof}
The positivity of the parameters $\l_j$ comes directly from the norms \eqref{embso} under which we are stating the
problem.
Furthermore, due to Young's inequality we find that
$$\tex{ \sqrt{\mu_1\mu_2} \int_{\mathbb{R}^N} |v_1|^2|v_2|^2  \leq \frac{1}{2}\left(\mu_1 \int_{\mathbb{R}^N} v_1^{4} +
    \mu_2 \int_{\mathbb{R}^N} v_2^{4}\right).}$$
Thus,
  $$\tex{\mu_1 \int_{\mathbb{R}^N} v_1^{4} +
    \mu_2 \int_{\mathbb{R}^N} v_2^{4} +   2\beta\int_{\mathbb{R}^N} |v_1|^2|v_2|^2  \geq 2(\sqrt{\mu_1\mu_2} +\beta) \int_{\mathbb{R}^N} |v_1|^2|v_2|^2,}$$
so that, it will be positive if and only if
$$\tex{\b >-\sqrt{\mu_1\mu_2},}$$
and therefore, proving the conditions in \eqref{poss2}.
\end{proof}

\begin{remark}
Note that due to Lemma\;\ref{rwell} the only considered possibility in order to have $r=r(v)$ well-defined will be \eqref{poss2}. Hence, the possible
condition when both terms in \eqref{rex2} are negative is neglected.
\end{remark}

Therefore, assuming that $\bv_c=(v_{1,c},v_{2,c})$ is a critical point of
$\mc{J}(r(\bv_c)\bv_c)$, thanks to the transformation carried out above, we again have that
a critical point $\bu_c=(u_{1,c},u_{2,c}) \in \mathbb{H}$, with $u_{j,c} \neq 0$ for $j=1$ or $j=2$, of
$\mc{J}$ is generated by $\bv_c$ through the expression
 $$
\bu_c=r_c\bv_c,
$$
 with $r_c$ defined by \eqref{rex2}.
Moreover, the different critical points of those fibering maps will
provide us with the critical points of the functional $\mc{J}(r(\bv)\bv)$ denoted by
\eqref{spfunc2}, and, hence, by construction, of the functional
$\mc{J}$ given by \eqref{mainfun}.

\vspace{0.2cm}

\noindent\underline{Lusternik-Schnirel'man analysis}. Similarly as performed for one single equation we apply
the topological method due to Lusternik-Schnirel'man in order to have an estimation of the number of solutions.

In this case the functional subset is denoted by
\begin{equation}
 \label{S0}
    \tex{ \mc{S}_{0,\l_1,\l_2}=\{\bu\in \mathbb{H}\,:\, \|\bu\|^2=1\}.}
\end{equation}
Again the critical points of the functional $\mc{J}$ are directly related with the
category $\g(\mc{S}_{0,\l_1,\l_2})$ of that functional subset \eqref{S0}, providing us with the number of critical points that
belong to the subset $\mc{S}_{0,\l_1,\l_2}$. Indeed,
the category $\g(\mc{S}_{0,\l_1,\l_2})$ is given by the number of eigenvalues
(with multiplicities) of the corresponding linear eigenvalue
problem satisfying:
 \be
 \label{rho12}
  \tex{
\g(\mc{S}_{0,\l_1,\l_2})= \sharp \{\ell_{\b,\l_j} > 0
\}, \quad \mbox{where}
}
\ee
\be
\label{rho22}
 \D^2 \psi_\b=(\ell_{\b,\l_j}-\l_j) \psi_\b, \quad \hbox{in}\quad \re^N \quad \hbox{and}
  \lim_{|x|\to \infty} \psi_\b(x)=0, \quad \hbox{with}\quad j=1,2.
\ee
For the particular case of the system \eqref{s1} we are able to ascertain
the existence of a countable family of solutions for the functional $\mc{J}$, but not to get any further information.
\begin{theorem}
$\mc{J}$ possesses an unbounded sequence
(family) of critical points.
\end{theorem}
\begin{remark} The proof follows the same argument performed in Theorem\;\ref{th21}  since  the linear part of system \eqref{nls}
consists of just two separated eigenvalue equations. Note that the coupling terms for system \eqref{nls} are of non-linear type.
\end{remark}
%For acknowledgements section, please don't number the section, please begin it with \section*{Acknowledgements}
% You may incorporate your references as follows in your main tex file.
% Using BibTex is not recommended but can be handled.
%%%%%%%%%%%%%%%%%%%%%%%%%%%%%%%%%%%%%%%%%%%%%%%%%%%%%%%%%%%%%%%%

\end{document}